\documentclass{commat}

\usepackage{%
    amsmath,
    amssymb,
    amssymb,
    amsfonts,
    color,
    enumitem,
    graphicx,
    hyperref,
    booktabs
}

\usepackage{microtype}
\usepackage[mode=buildnew]{standalone} % requires --shell-escape
\usepackage{tikz}

\newcommand{\IZ}{{\mathbb{Z}}} 
\newcommand{\IZn}{\IZ^n}
\newcommand{\IR}{{\mathbb{R}}} 
\newcommand{\IRn}{\IR^n}

\newcommand{\sR}{\mathcal{R}}

\DeclareMathOperator{\cone}{cone}
\DeclareMathOperator{\conv}{conv}

\DeclareMathOperator{\bd}{bd}
\DeclareMathOperator{\interior}{int}

\newcommand{\stwo}{{\mathcal{S}}^{2}} 
\newcommand{\sn}{{\mathcal{S}}^{n}} 
\newcommand{\sngeo}{{\mathcal{S}}^{n}_{\geq0}} 
\newcommand{\sngo}{{\mathcal{S}}^{n}_{>0}}

\newcommand{\nonn}{{\mathcal{N}^{n}}}
\newcommand{\IZngeo}{\IZ_{\geq0}^n}
\newcommand{\IZnmo}{\IZn\setminus\{0\}}
\newcommand{\IZngeomo}{\IZngeo\setminus\{0\}}
\newcommand{\IRngeo}{\IR_{\geq0}^n}
\newcommand{\IRgeo}{\IR_{\geq0}}
 
\newcommand{\cop}{{\mathcal{COP}}} 
\newcommand{\coptwo}{\cop^2}
\newcommand{\copn}{{\mathcal{COP}}^{n}} 
\newcommand{\cp}{{\mathcal{CP}}}
\newcommand{\cpn}{{\mathcal{CP}}^{n}} 

\DeclareMathOperator{\Min}{Min}
\DeclareMathOperator{\minC}{min_{\cop}}
\DeclareMathOperator{\MinC}{Min_{\cop}}

\newcommand{\gl}{\operatorname{GL}}
\newcommand{\glnz}{\gl_n (\IZ)}
\newcommand{\glnr}{\gl_n (\IR)}

\newcommand{\lfp}{\mathcal{P}}
\newcommand{\lfpfacef}{\mathcal{F}}
\newcommand{\lfpfaceg}{\mathcal{G}}
\newcommand{\polyh}{\mathsf{P}}
\newcommand{\pfacef}{\mathsf{F}}
\newcommand{\pfaceg}{\mathsf{G}}

\title{%
    Perfect copositive matrices
    }

\author{%
    Valentin Dannenberg and Achill Sch\"urmann
    }

\authorinfo[%
    V. Dannenberg]{% Please, use "F. Name" format
    University of Rostock, 18051 Rostock, Germany}{%
    valentin.dannenberg2@uni-rostock.de
    }

\authorinfo[%
    A. Sch\"urmann]{% Please, use "S. Name" format
    University of Rostock, 18051 Rostock, Germany}{%
    achill.schuermann@uni-rostock.de
    }

\abstract{%
	In this paper we give a first study of perfect copositive \( n\times n \) matrices.
	They can be used to find rational certificates for completely positive matrices.  
	We describe similarities and differences to classical perfect, positive definite matrices.
	Most of the differences occur only for \( n\geq3 \), where we find for instance lower rank and 
    indefinite perfect matrices.
	Nevertheless, we find for all $n$ that for every classical perfect matrix  
	there is an arithmetically equivalent one which is also perfect copositive.
	Furthermore we study the neighborhood graph and polyhedral structure of perfect copositive 
    matrices. 
	As an application we obtain a new characterization of the
	cone of completely positive matrices: It is equal to the set of nonnegative matrices having a
	nonnegative inner product with all perfect copositive matrices.
    }



\VOLUME{31}
\NUMBER{2}
\YEAR{2023}
\firstpage{137}
\DOI{https://doi.org/10.46298/cm.11141}

\begin{document}

\section{Introduction and Overview}

Perfect matrices and their respective perfect quadratic forms have been investigated intensively since the
19th century.
This was mainly driven by a~number theoretical interest in connection with the arithmetical minimum
of a~positive definite quadratic form.
A milestone in the classical theory is a~reduction theory based on perfect forms due to
Voronoi~\cite{voronoi-1907}.
His theory has been generalized in different directions (see~\cite{martinet-2002}), with important
steps in particular by Koecher~\cite{koecher-1960, koecher-1960a} and
Opgenorth~\cite{opgenorth-2001}.
While Koecher generalized the classical theory for the cone of positive definite matrices to a~
theory within any selfdual convex cone of symmetric matrices, Opgenorth took the theory to any
pair of dual convex cones.

In the recent paper~\cite{dsv-2021} a~generalized theory of perfect matrices in
the cone of copositive matrices is used for an algorithmic solution of a~fundamental problem in
copositive optimization:
With a~Voronoi-like simplex-type algorithm it is possible to obtain a~rational certificate for a~
given symmetric matrix which is completely positive.
The central object in the algorithm is the neighboring graph of generalized perfect matrices,
which we call \emph{perfect copositive matrices} here.

In this paper we give a~first systematic study of perfect copositive matrices.
In particular, we show where there are similarities and differences to the classical theory of
perfect matrices.

Our paper is organized as follows:
In Section~\ref{sec:classical} we review the most important notions from the classical theory we
need for our generalization, defined in Section~\ref{sec:copositive}.
Sections~\ref{sec:nequal2} and~\ref{sec:newphenomenaN3} are devoted to a~closer look at
\(n\times n \) matrices for \(n=2 \) and \(n=3 \):
We find that a~lot of differences between the classical and the copositive theory occur only for
\(n \geq 3 \).
In Section~\ref{sec:embedding} we show that the classical theory can, in essence, be embedded in
the copositive theory.
Finally, in Section~\ref{sec:polyrepr} we examine the underlying polyhedral structure
and obtain a~new characterization for the cone of completely positive matrices.

\section{Perfect Positive Definite Matrices}\label{sec:classical}

The classical theory is concerned with positive definite \(n\times n \) matrices $Q$ and their
\emph{arithmetical minimum}
\[
	\min Q=\min\big\{ Q[v] :\ v\in\IZnmo \big\}.
\]
Here and in what follows we use the notation \(Q[v] \) for \(v^\intercal Qv \).
Classically, perfect matrices are symmetric positive definite matrices that are uniquely determined
by their arithmetical minimum and the integer vectors
\[
	\Min Q=\big\{ v\in\IZn :\ Q[v]=\min(Q) \big\}
\]
attaining it.
For a~fixed integer vector $v$ the condition \(Q[v]=c \), for a~constant $c$, is linear for
the variable $Q$ in the space \(\sn \) of symmetric \(n \times n \) matrices.
Therefore, it is sufficient to have \(\dim\sn=\binom{n+1}{2} \) linear independent
conditions in \(\sn \) through integer vectors $v$ in order to obtain a~perfect matrix $Q$
as a~unique solution of a~system of linear equations.

Using linearities, a~milestone in the classical theory is a~polyhedral reduction theory for
positive definite forms and matrices due to Voronoi~\cite{voronoi-1907}.
A central result for his theory is the fact that in any given dimension $n$ and for a~fixed
arithmetical minimum, there exist only finitely many perfect matrices up to
\emph{arithmetical equivalence}.
Here, two symmetric matrices \(Q_1 \) and \(Q_2 \) are called arithmetically equivalent,
if there exists a~unimodular matrix \(U\in\glnz \) with \(Q_2=U^\intercal Q_1U \).
Classifying perfect matrices up to arithmetical equivalence had already attracted Lagrange
(\(n=2 \)), Gauß (\(n=3 \)) and others.
Based on his theory Voronoi gave an algorithm for the classification for any given dimension $n$.
Still being the only known approach for classification in arbitrary dimension, it has been used by
several authors, see~\cite{dsv-2007},~\cite{martinet-2002},~\cite{schuermann-2009a},~\cite{vanwoerden-2018}.
Due to the difficulty of polyhedral representation conversions the classification is nevertheless
still open for all \(n\geq 9 \).

Voronoi's algorithm in modern terms can be described as a~graph traversal search on the
vertex-edge graph of the \emph{Ryshkov-Polyhedron}
\begin{equation}\label{eq:ryshkov-poly}
	\sR = \big\{ Q\in\sn :\ Q[v]\geq1 \text{ for all } v\in\IZnmo \big\}.
\end{equation}
Its vertices correspond to perfect matrices of arithmetical minimum $1$.
For every perfect matrix $P$, respectively for every vertex of the Ryshkov polyhedron,
its \emph{Voronoi cone} is the convex cone
\[
	\cone\big\{ xx^\intercal :\ x\in\Min P \big\}
	=\left\{ \sum_{i=1}^{m}\alpha_ix_ix_i^\intercal:\
        m\in\mathbb{N}, \alpha_i\in\IRgeo, x_i\in\Min P
     \right\},
\]
generated by rank-$1$-matrices given by its minimal vectors.
From general duality theory we can deduce that the union of Voronoi cones of perfect matrices gives
a polyhedral decomposition for the rational closure of the cone of positive definite matrices.
We refer to~\cite{schuermann-2009} for more details and proofs of the classical theory.

\section{Perfect copositive matrices}\label{sec:copositive}

Over the years, different variants of Voronoi's theory and algorithm have been developed.
In the context of number theory, in particular the works of
Koecher~\cite{koecher-1960},~\cite{koecher-1960a} and Opgenorth~\cite{opgenorth-2001} have provided
significant contributions.
In them, the cone of positive definite matrices is replaced by any selfdual convex cone,
respectively a~pair of convex cones dual to each other.

Building on these ideas, a~variant of the Ryshkov polyhedron was defined and used in~\cite{dsv-2017}
and~\cite{dsv-2021} for a~fundamental problem in copositive optimization
(see~\cite{bsu-2012},~\cite{duer-2010}), namely for obtaining a~rational certificate for
\emph{\(\cp \)-factorizations} (see~\cite{bs-2018},~\cite{gd-2020},~\cite{sdb-2018}).
The \emph{cone of completely positive matrices} is
\begin{align*}
	\cpn
	& = \cone\left\{ xx^\intercal :\ x\in\IRngeo \right\} \\
	& = \left\{ \sum_{i=1}^m \alpha_ix_ix_i^\intercal :\
			 m\in\mathbb{N},\alpha_i\in\IRgeo,x_i\in\IRngeo
		\right\}.
\end{align*}
A rational \(\cp \)-factorization of a~matrix \(Q\in\cpn \) is a~decomposition
\[
	Q = \sum_{i=1}^m \alpha_ix_ix_i^\intercal
\]
with \(\alpha_i\in\mathbb{Q} \) and \(x_i\in\IZn \).
It is still an open problem  whether every rational completely positive matrix possesses such a
factorization~\cite{bds-2015}, though every rational matrix in the interior of \(\cpn \)
can be factorized in such a~way~\cite{dsv-2017}.
However, the task to find a~rational \(\cp \)-factorization is known to be a~difficult
problem, since it gives a~certificate for \(Q\in\cpn \)~\cite{dg-2014}.

With the standard inner product
\(\left< A, B \right> = \operatorname{Trace}(AB) = \sum_{i,j=1}^n A_{ij}B_{ij} \)
on \(\sn \) we obtain the dual cone of \(\cpn \), the \emph{cone of copositive matrices}
\begin{align*}
	\copn = {(\cpn)}^*
		& = \big\{ B\in\sn :\ \left< B,A \right>\geq0 \text{ for all } A\in\cpn \big\} \\
		& = \big\{ B\in\sn :\ B[x]\geq0 \text{ for all } x\in\IRngeo \big\}.
\end{align*}
Note that also \(\cpn={(\copn)}^* \) holds.

This duality relation can be used to give a~certificate for \(Q\notin\cpn \),
by providing a~suitable copositive matrix $B$ with \(\langle B,A\rangle<0 \).
In~\cite{dsv-2021} the authors conjecture that a~suitable variant of their algorithm can find such
a matrix $B$ (a certificate for \(Q\not\in\cpn \)), or provide a~\(\cp \)-factorization of
a given rational input matrix $Q$ (a certificate for \(Q\in\cpn \)).

The basis for algorithms as developed in~\cite{dsv-2021} is a~generalization of the theory of
perfect matrices, in particular its neighborhood graph.
This graph is naturally obtained by generalizing the Ryshkov polyhedron of the classical
theory (see~\cite{schuermann-2009}).
For its definition in the cone of copositive matrices we need to generalize the arithmetical
minimum to the copositive setting:
The \emph{copositive minimum} of a~symmetric matrix $B$ is defined by
\[
	\minC B = \inf\big\{ B[v] :\ v\in\IZngeomo \big\},
\]
and the set of vectors attaining it is
\[
	\MinC B = \big\{ v\in\IZngeo :\ B[v]=\minC B \big\}.
\]
A vector \(v\in\MinC B \) is also called a~minimal vector of $B$.
If $B$ lies in the interior of the copositive cone, it is also called \emph{strictly copositive}.
For such matrices, the copositive minimum is particularly well-behaved,
as we have the following result.
\begin{lemma}[{\cite[Lemma 2.2]
{dsv-2021}}] For a~strictly copositive matrix $B$ we have \(\minC B>0 \) and \(|\MinC B|<\infty \).
\end{lemma}
Generalizing the notion of perfectness, a~strictly copositive matrix is called
\emph{perfect copositive} if it is uniquely determined by its copositive minimum and the set of
vectors attaining it.
The \emph{copositive Ryshkov polyhedron} is defined as
\[
	\sR_\cop = \big\{ B\in\sn :\ B[v]\geq1 \text{ for all } v\in\IZngeomo \big\}
\]
and its vertices correspond to the perfect copositive matrices with copositive minimum $1$.

For every perfect copositive matrix $P$, respectively for every vertex of the copositive
Ryshkov polyhedron, its \emph{copositive Voronoi cone} is the convex cone
\[
	\mathcal{V}(P) = \cone\big\{ xx^\intercal :\ x\in\MinC P \big\}.
\]
As in the classical case, by using duality, we can deduce that the union of copositive Voronoi
cones of perfect copositive matrices gives a~polyhedral decomposition of \(\cpn \).

Traversing the vertices of the copositive Ryshkov polyhedron,
resp.\ the neighborhood graph of perfect copositive matrices,
can be done as in the classical theory.
We describe the procedure in short:
Starting with a~vertex of the copositive Ryshkov polyhedron \(\sR_\cop \),
that is, with a~perfect copositive matrix $P$ with copositive minimum \(\minC P=1 \),
we can find its \emph{contiguous (neighboring) perfect copositive matrices}
by considering generators of extreme rays of
\begin{align*}
	{(\mathcal{V}(P))}^*
    & = \left\{ B\in\sn:\ \langle B, X \rangle\geq 0 \text{ for all } X\in\mathcal{V}(P) \right\} \\
	& = \left\{ B\in\sn:\ B[x]\geq 0 \text{ for all } x\in\MinC P \right\},
\end{align*}
the dual of the copositive Voronoi cone of $P$.
Note that \(\mathcal{V}(P)\) and therefore its dual are both polyhedral cones,
that is, both are generated by finitely many extreme rays.

The matrix $Q$ is a~contiguous perfect copositive matrix of $P$, if there exist
a generator $R$ of an extreme ray of \({(\mathcal{V}(P))}^* \) and \(\lambda>0 \)
such that \(Q=P+\lambda R \) has copositive minimum $1$
and $Q$ is perfect copositive itself.
In this case, \(\minC(P+\mu R)<1 \) for all \(\mu>\lambda \), whenever \(P+\mu R \) is copositive.
The question of how to practically find \(\lambda \) and hence $Q$ is described in detail in~\cite[Section~3]{dsv-2021}.
It turns out that for some perfect copositive matrices and certain extreme ray generators $R$,
\(\minC(P+\mu R)=1 \) for all \(\mu\geq 0 \).
That is, there is no contiguous perfect copositive matrix of $P$ in the direction of $R$.
This situation does not occur in the classical theory and is an interesting new phenomenon
for the copositive case.
We show in Section~\ref{sec:polyrepr} that it can only occur for very special $R$.

It was shown in~\cite[Lemma 2.5]{dsv-2021},
that as a~possible starting perfect copositive matrix in dimension $n$ we can choose
\(\frac12 Q_{\mathsf{A}_n} \) with
\begin{equation}\label{eqn:initial_perfect_form}
    Q_{\mathsf{A}_n} =
    \begin{pmatrix}
    \phantom{-}2 & -1 & 0 & \cdots & 0 \\
    -1 & \phantom{-}2 & \ddots & \ddots & \vdots \\
    0 & \ddots & \ddots & \ddots & 0 \\
    \vdots & \ddots & \ddots & \phantom{-}2 & -1 \\
    0 & \cdots & 0 & -1 & \phantom{-}2
    \end{pmatrix}
.
\end{equation}
The matrix \(Q_{\mathsf{A}_n} \) is a~Gram matrix of the root lattice \(\mathsf{A}_n \).
Its copositive minimum \(\minC Q_{\mathsf{A}_n} = 2 \) is attained by the vectors
\[
    \MinC Q_{\mathsf{A}_n}=\left\{ \sum_{i=j}^k e_i : 1 \leq j \leq k \leq n \right\},
\]
where \(e_i \) is the $i$-th standard unit basis vector of \(\IRn \).
Note that \(Q_{\mathsf{A}_n} \) is an example of a~
classical perfect matrix that is also perfect copositive.
The reason why this works is that \(Q_{\mathsf{A}_n} \) is the Gram matrix of a~basis of
\(\mathsf{A}_n \), which is also a~\emph{fundamental system of roots} of the root system generated
by the shortest vectors of \(\mathsf{A}_n \), see~\cite[Section 1.4]{ebeling-2002}.

Generalizing this idea, one can find Gram matrices of the (perfect) root lattices
\(D_n, E_6, E_7 \) and \(E_8 \), which are perfect copositive as well.
In~\cite{martinet-2002} these bases are also called the \emph{Coxeter-Bases} of the root lattices.

In fact, whenever it is the case that for a~classical perfect matrix $P$
the set \(\Min P \) consists only of vectors \(\pm x \) with \(x\in\IZngeo \), $P$ is also
copositive perfect.
We use this in Section~\ref{sec:embedding} to show that for every classical perfect positive
definite matrix, there exists an arithmetically equivalent one which is perfect copositive.
\subsection*{Where to look for perfect copositive matrices?}
Let \(\nonn \) denote the cone of symmetric, componentwise nonnegative \(n\times n\) matrices, and
\(\sngeo \) the cone of symmetric positive semidefinite \(n\times n\) matrices.
It is known that \(\sngeo+\nonn\subseteq\copn \), with equality if \(n\leq4 \).
This gives us several ``natural components'' of \(\copn \),
where we can look for perfect copositive matrices.
In~\eqref{eqn:initial_perfect_form} we have already seen that there are positive definite matrices,
which are perfect copositive.
The same however cannot be said about nonnegative matrices (if \(n>1 \)).
\begin{lemma}\label{lem:nonnegnonperf}
    \(P\in\nonn \) is perfect copositive if and only if \(n=1 \)
\end{lemma}
\begin{proof}
    If \(n=1 \), the only perfect copositive matrix (up to scaling) is \((1) \in\mathcal{N}^1 \).
    On the other hand, let \(P=(p_{ij}) \).
    For any vector \(x\in\IZngeo\setminus\{0\} \) with \(x_i\neq0 \) we have
    \[
        P[x]=\sum_{i,j=1}^{n}p_{ij}x_ix_j\geq p_{ii}x_i^2 \geq p_{ii}=P[e_i],
    \]
    with strict inequality in the first part, if at least two entries of $x$ are nonzero
    and in the second part, if \(x_i \geq 2 \).
    Hence, we see that
    \[
        \MinC P=\left\{ e_i :\ p_{ii}=\min_{j=1,\dots,n}p_{jj} \right\}.
    \]
    Therefore \(\MinC P \) contains at most $n$ elements, but \(n<\binom{n+1}{2} \),
	whenever \(n\geq2 \).
	So $P$ cannot be perfect if \(n > 1 \).
\end{proof}
The remaining components of \(\copn \) we look at are \(\bd\sngeo \),
 \((\sngeo+\nonn)\setminus(\sngeo\cup\nonn) \) and \(\copn\setminus(\sngeo+\nonn) \).
In Section~\ref{sec:newphenomenaN3} we show that in all of these components there are perfect
copositive matrices.

\subsection*{Symmetries}
In the previous section we discussed a~notion of equivalence of classical perfect matrices,
induced by \(\glnz \).
It turns out that the (infinite) group \(\glnz \) is the symmetry group of the Ryshkov polyhedron.
Here, by symmetry group of a~subset $S$ of \(\sn \) we mean the set
\[
    \operatorname{Sym}(S)=\left\{ U\in\glnr: U^\intercal SU=S \right\}.
\]
In the copositive case it is known (see~\cite{sst-2013},~\cite{shitov-2021}) that the
symmetry group of \(\copn \) is the group of nonnegative generalized permutation matrices,
i.e.\ the set of all \(n\times n\) matrices of the form \(PD \) with a~permutation matrix $P$ and
a diagonal matrix $D$ with positive diagonal.
This leaves us with just a~finite group acting on the perfect copositive matrices and on \(\sR_\cop \).
\begin{lemma}\label{lem:symmetrygroup}
	The symmetry group of the copositive Ryshkov polyhedron is
	
\[
		\operatorname{Sym}(\sR_\cop)=\operatorname{Sym}(\copn)\cap\glnz\cong S_n,
\]
	with \(S_n \) being the symmetric group of order $n$.
	In other words, the symmetries of \(\sR_\cop \) are the \(n\times n\) permutation matrices.
\end{lemma}
\begin{proof}
    The inclusion \(\supseteq \) follows from a~straightforward calculation.
    For the inclusion \(\subseteq \), note that~\cite[Section 2.2]{dsv-2021}
    \[
        \interior\copn = \cone\sR_\cop\setminus\{0\}.
    \]
    Thus
    \[
        \operatorname{Sym}(\sR_\cop)\subseteq
        \operatorname{Sym}(\cone\sR_\cop\setminus\{0\})=
		\operatorname{Sym}(\interior\copn)\subseteq
        \operatorname{Sym}(\copn).
    \]
    Therefore, every \(U \in\operatorname{Sym}(\sR_\cop) \) has the form \(U=PD \)
    for a~permutation matrix $P$ and a~diagonal matrix $D$ as above.
    Since such a matrix~$U$ has to act on \(\IZngeo \), $D$ is equal to the identity matrix,
    giving us the desired result.
\end{proof}
Just as in the classic theory, we can now introduce a~notion of equivalence of two perfect
copositive matrices.
Here we say that the perfect copositive matrices \(Q_1 \) and \(Q_2 \) are equivalent
(in the copositive sense) if there exists a~permutation matrix $P$,
such that \(Q_2 = P^\intercal Q_1P \).
Note that, in contrast to the classic theory, there are infinitely many copositive perfect matrices,
even up to equivalence.
Thus we cannot hope for a~Voronoi-type result in the copositive case.

\section{The case \texorpdfstring{\(n=2 \)}{n=2}}\label{sec:nequal2}

While the structure of the copositive cone is very difficult for general $n$, for $n=2$,
it is comparatively simple, as \(\coptwo=\stwo_{\geq0}\cup\mathcal{N}^2 \), see~\cite{py-1993}.
Together with Lemma~\ref{lem:nonnegnonperf} this gives us a~classification of \(2\times 2 \)
perfect copositive matrices.
\begin{theorem}\label{thm:characterizeN2}
	Perfect copositive \(2 \times 2 \) matrices are precisely the classical perfect matrices,
    having negative nondiagonal entries.
\end{theorem}

Another advantage for \(n=2 \) is that we can visualize parts of the $3$-dimensional
Ryshkov-polyhedron to get a~clear impression of what is going on.
A standard way of vizualization is a~projection to a~plane of constant trace
(see Figure~\ref{fig:rcop-2}). Here, the vertex \(\frac12 Q_{\mathsf{A}_2} \)
of the Ryshkov polyhedron shows an essential difference between the classical
case and the copositive Ryshkov polyhedron.
It is connected to two contiguous classical perfect matrices, which are also
perfect copositive, but it is also connected to an extreme ray
\begin{equation}\label{eqn:extreme_rayn2}
    \left\{ \frac12 Q_{\mathsf{A}_2}+\lambda\cdot E_{12} :\ \lambda\geq0 \right\}
\end{equation}
of the Ryshkov polyhedron, where
\(E_{12}=
\begin{pmatrix}
0 & 1 \\ 1 & 0
\end{pmatrix}
\).
By the results obtained in Section~\ref{sec:polyrepr} this is the only extreme ray for \(n=2 \)
and the three-dimensional Ryshkov polyhedron is equal to
\begin{align}\label{eq:rcop-2-poly}
	\sR_\cop
	&=\conv\big\{ P\in\coptwo :\ P \text{ perfect copositive, } \minC P=1 \big\} \\
	&\quad +\cone\big\{ E_{12} \big\} \nonumber
\end{align}
Such extreme rays do not exist for the classical Ryshkov polyhedron.
There is in fact a~contiguous classical perfect matrix on
the ray~\eqref{eqn:extreme_rayn2} for \(\lambda=1 \), that is, the matrix
\(\frac12 \cdot
\begin{pmatrix}
2 & 1 \\ 1 & 2
\end{pmatrix}
\).
It is -- as are all vertices of the classical Ryshkov polyhedron for \(n=2 \) --
arithmetically equivalent to
\(\frac12 Q_{\mathsf{A}_2} \).
This can for instance be seen by the relation
\(
\begin{pmatrix}
2 & 1 \\ 1 & 2
\end{pmatrix}
=U^\intercal Q_{\mathsf{A}_2}U \)
for the unimodular matrix
\(U=
\begin{pmatrix}
1 & 0 \\ 0 & -1
\end{pmatrix}
\).

In fact, the linear map \(Q\mapsto U^\intercal QU \),
with $U$ as above, maps every perfect copositive matrix $Q$ to a~
perfect positive definite, but nonnegative matrix and vice versa.
So in some sense ``exactly half'' of the classical perfect matrices
are perfect copositive, visible in Figure~\ref{fig:rcop-2} as a~mirror symmetry
along the $x$-axes of grey and black parts of the classical Ryshkov polyehdron.
The other symmetry visible in Figure~\ref{fig:rcop-2} (mirror along the $y$-axis)
is in a~similar way given by the linear map \(Q\mapsto U^\intercal QU \),
but with \(U=
\begin{pmatrix}
0 & 1 \\ 1 & 0
\end{pmatrix}
\),
permuting coordinates of minimal vectors.
Note that this leaves the vertex \(\frac12 Q_{\mathsf{A}_2} \) as well as
the attached extreme ray invariant.

\begin{figure}
	\centering
	\begin{tikzpicture}
	\tikzstyle{cop} = [circle,fill=black,inner sep=0pt,minimum size=3pt]
	\tikzstyle{bdcop} = [circle,fill=black,inner sep=0pt,minimum size=2pt]
	\tikzstyle{pqf} = [circle,fill=gray,inner sep=0pt,minimum size=3pt]
	\tikzstyle{bdpqf} = [circle,fill=gray,inner sep=0pt,minimum size=2pt]

	%s2geo
	\begin{scope}
		\clip (-4, 4) rectangle (4, 0);
		\draw[lightgray] (0, 0) circle(4);
	\end{scope}
	% cop2
	\begin{scope}
		\clip (-4, -4) rectangle (4, 0);
		\draw (0, 0) circle(4);
	\end{scope}
	\draw[thin] (-4, 0) -- (-4, 5);
	\draw[thin] (4, 0) -- (4, 5);
    
	% COP-perfect forms
	\node[label={[label distance=0.1em]20:$Q_{A_{2}}$}, cop] (Q_A_2) at (0, -2) {};
	\node[cop] (1) at (2, -3) {};
	\node[cop] (1s) at (-2, -3) {};
	\node[cop] (2) at (3, -2.5) {};
	\node[cop] (2s) at (-3, -2.5) {};
	\node[cop] (3) at (1.6, -3.6) {};
	\node[cop] (3s) at (-1.6, -3.6) {};
	\node[bdcop] (4) at (24/7, -2) {};
	\node[bdcop] (4s) at (-24/7, -2) {};
	\node[bdcop] (5) at (32/11, -30/11) {};
	\node[bdcop] (5s) at (-32/11, -30/11) {};
	\node[bdcop] (6) at (1.2, -3.8) {};
	\node[bdcop] (6s) at (-1.2, -3.8) {};
	\node[bdcop] (7) at (24/13, -46/13) {};
	\node[bdcop] (7s) at (-24/13, -46/13) {};
    
	% edges of R_COP
	\draw[thin] (Q_A_2) -- (1);
	\draw[thin] (Q_A_2) -- (1s);
	\draw[thin] (1) -- (2);
	\draw[thin] (1) -- (3);
	\draw[thin] (1s) -- (2s);
	\draw[thin] (1s) -- (3s);
	\draw[thin] (2) -- (4);
	\draw[thin] (2) -- (5);
	\draw[thin] (2s) -- (4s);
	\draw[thin] (2s) -- (5s);
	\draw[thin] (3) -- (6);
	\draw[thin] (3) -- (7);
	\draw[thin] (3s) -- (6s);
	\draw[thin] (3s) -- (7s);
    
	 % classically perfect forms
	\node[pqf] (Q_A_2_a) at (0, 2) {};
	\node[pqf] (1_a) at (2, 3) {};
	\node[pqf] (1s_a) at (-2, 3) {};
	\node[pqf] (2_a) at (3, 2.5) {};
	\node[pqf] (2s_a) at (-3, 2.5) {};
	\node[pqf] (3_a) at (1.6, 3.6) {};
	\node[pqf] (3s_a) at (-1.6, 3.6) {};
	\node[bdpqf] (4_a) at (24/7, 2) {};
	\node[bdpqf] (4s_a) at (-24/7, 2) {};
	\node[bdpqf] (5_a) at (32/11, 30/11) {};
	\node[bdpqf] (5s_a) at (-32/11, 30/11) {};
	\node[bdpqf] (6_a) at (1.2, 3.8) {};
	\node[bdpqf] (6s_a) at (-1.2, 3.8) {};
	\node[bdpqf] (7_a) at (24/13, 46/13) {};
	\node[bdpqf] (7s_a) at (-24/13, 46/13) {};

	% edges of R
	\draw[thin,gray] (Q_A_2_a) -- (1_a);
	\draw[thin,gray] (Q_A_2_a) -- (1s_a);
	\draw[thin,gray] (1_a) -- (2_a);
	\draw[thin,gray] (1_a) -- (3_a);
	\draw[thin,gray] (1s_a) -- (2s_a);
	\draw[thin,gray] (1s_a) -- (3s_a);
	\draw[thin,gray] (2_a) -- (4_a);
	\draw[thin,gray] (2_a) -- (5_a);
	\draw[thin,gray] (2s_a) -- (4s_a);
	\draw[thin,gray] (2s_a) -- (5s_a);
	\draw[thin,gray] (3_a) -- (6_a) {};
	\draw[thin,gray] (3_a) -- (7_a) {};
	\draw[thin,gray] (3s_a) -- (6s_a) {};
	\draw[thin,gray] (3s_a) -- (7s_a) {};
    
	 % extreme ray E_12
	\draw[gray] (Q_A_2) -- (Q_A_2_a);
	\draw[->] (Q_A_2_a) -- (0, 5);

	% labels
	\node (R_COP) at (0,-3) {$\mathcal{R}_\mathcal{COP}$};
	\node[lightgray,left] (R) at (0,0) {$\mathcal{R}$};
	\node (COP) at (0, -4.5) {$\mathcal{COP}^2$};
	\node[lightgray,left] (S2geo) at (0, 4.5) {$\mathcal{S}^2_{\geq 0}$};
\end{tikzpicture}
	\caption{The copositive (black) and classical (black and grey) Ryshkov polyhedron for the case
		 \(n=2 \), projected onto the hyperplane \(\operatorname{Trace}Q=2 \).
    	Cartesian axes are \(x=q_{11}-q_{22} \) and \(y=q_{12} \) for \(Q=(q_{ij})\in\stwo \).
	}%
	\label{fig:rcop-2}
\end{figure}
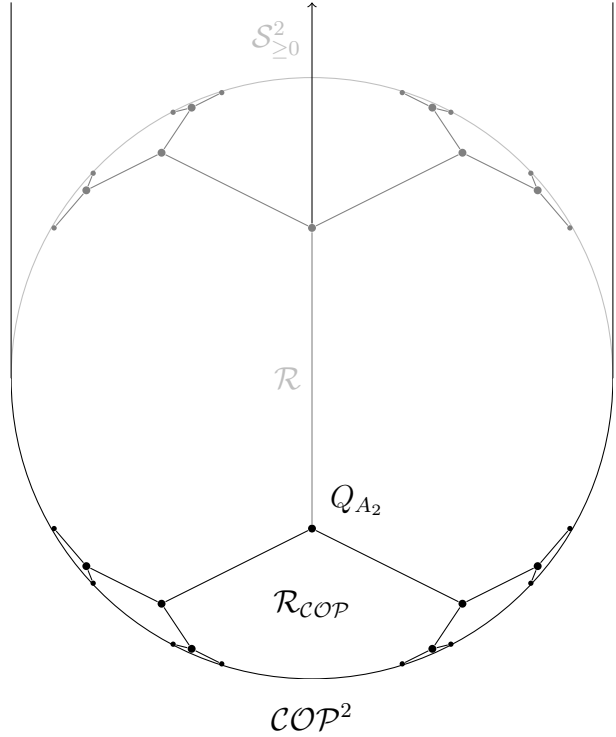

We close this section by describing how understanding contiguous perfect matrices and traversing
the neighborhood graph is possible by just looking at the changes of sets of minimal vectors.
Considering the matrix \(Q_{\mathsf{A}_2} \) for instance we have
\[
    \MinC Q_{\mathsf{A}_2} =
    \left\{
    \begin{pmatrix}
1 \\ 0
\end{pmatrix}
,
    \begin{pmatrix}
0 \\ 1
\end{pmatrix}
,
    \begin{pmatrix}
1 \\ 1
\end{pmatrix}
    \right\}.
\]
These three vectors define \(Q_{\mathsf{A}_2} \) uniquely.
By removing one of the vectors, we obtain a~ray $R$ so that
\(Q=Q_{\mathsf{A}_2}+\lambda R \), for small enough \(\lambda>0 \),
has \(\MinC Q \) consisting of the two remaining vectors.
As a~special case, the set
\(\left\{
\begin{pmatrix}
1 \\ 0
\end{pmatrix}
,
\begin{pmatrix}
0 \\ 1
\end{pmatrix}
\right\} \)
is attained by \(Q=Q_{\mathsf{A}_2}+\lambda E_{12} \) for all \(\lambda>0 \).
For the other two choices of removing one vector from \(\MinC Q_{\mathsf{A}_2} \),
the chosen vector is replaced by another one to obtain a~contiguous perfect copositive matrix.
For instance,
\[
    \text{ replacing }
\begin{pmatrix}
1 \\ 0
\end{pmatrix}
\text{ with }
    \begin{pmatrix}
1 \\ 2
\end{pmatrix}
    \text{ yields the perfect neighbor }
    \begin{pmatrix}
        6 & -3\\
        -3 & 2
    \end{pmatrix}
,
\]
\[
    \text{and replacing }
    \begin{pmatrix}
0 \\ 1
\end{pmatrix}
\text{ with }
    \begin{pmatrix}
2 \\ 1
\end{pmatrix}
    \text{ yields the perfect neighbor }
    \begin{pmatrix}
        2 & -3\\
        -3 & 6
    \end{pmatrix}
.
\]
In fact, for a~perfect copositive matrix $P$ with
\[
    \MinC P=\left\{
    \begin{pmatrix}
a~\\ b
\end{pmatrix}
,
    \begin{pmatrix}
c \\ d
\end{pmatrix}
,
    \begin{pmatrix}
e \\ f
\end{pmatrix}
    \right\}
\]
by replacing the vector \({(e, f)}^\intercal \), we obtain the contiguous perfect matrix $N$ with
\[
    \MinC N=\left\{
    \begin{pmatrix}
a~\\ b
\end{pmatrix}
,
    \begin{pmatrix}
c \\ d
\end{pmatrix}
,
    \begin{pmatrix}
a+c \\ b+d
\end{pmatrix}
    \right\}\neq\MinC P.
\]
We refer to~\cite{dsv-2021} for details and for connections to
approximations of irrationals by rational numbers using continued fractions.

\section{New Phenomena for \texorpdfstring{\(n \geq 3 \)}{n>=3}}\label{sec:newphenomenaN3}

A simple characterization, just as in Theorem~\ref{thm:characterizeN2} for \(n=2 \),
is no longer true for \(n \geq 3 \).
In fact, we find several new phenomena which do not occur for classical perfect matrices.
In particular, we will see that for \(n\geq 3 \) there are perfect copositive matrices in
all the remaining components of \(\copn \) mentioned in Section~\ref{sec:copositive}.

Let us start by looking at the neighborhood graph for \(n=3 \), starting with~\(Q_{\mathsf{A}_3} \).
We have
\[
    \MinC (Q_{\mathsf{A}_3})=
    \left\{
    \begin{pmatrix}
1 \\ 0 \\ 0
\end{pmatrix}
,
    \begin{pmatrix}
0 \\ 1 \\ 0
\end{pmatrix}
,
    \begin{pmatrix}
0 \\ 0 \\ 1
\end{pmatrix}
,
    \begin{pmatrix}
0 \\ 1 \\ 1
\end{pmatrix}
,
    \begin{pmatrix}
1 \\ 1 \\ 0
\end{pmatrix}
,
    \begin{pmatrix}
1 \\ 1 \\ 1
\end{pmatrix}
    \right\}
\]
and every choice of one vector to be replaced leads to a~
contiguous perfect copositive matrix or an extreme ray of the copositive Ryshkov polyhedron.
In fact, removing the last vector (consisting of 1 in each coordinate) yields the extreme ray
\[
    \left\{ \frac12 Q_{\mathsf{A}_3}+\lambda\cdot E_{13} :\ \lambda\geq0 \right\},
\]
where \(E_{13} =
\begin{pmatrix}
0 & 0 & 1 \\ 0 & 0 & 0 \\ 1 & 0 & 0
\end{pmatrix}
\).
All other choices of removing one vector from \(\MinC Q_{\mathsf{A}_3} \) lead to a~contiguous
copositive perfect matrix.

Two of them, (replacing \({(0,1,1)}^\intercal \) or \({(1,1,0)}^\intercal\) by \({(1,0,1)}^\intercal \))
\[
    \begin{pmatrix}
2 & -1 & -1 \\ -1 & 2 & 0 \\ -1 & 0 & 2
\end{pmatrix}
    \quad\mbox{and}\quad
    \begin{pmatrix}
2 & 0 & -1 \\ 0 & 2 & -1 \\ -1 & -1 & 2
\end{pmatrix}
,
\]
are equivalent to \(Q_{\mathsf{A}_3} \) (in the copositive sense).
Hence these two neighbors have, up to a~coordinate permutation,
the same neighborhood as \(Q_{\mathsf{A}_3} \).

Two other neighbors of \(Q_{\mathsf{A}_3} \)
(replacing \(e_1 \) or \(e_3 \) by \({(1,2,1)}^\intercal \)) are
\[
    \begin{pmatrix}
4 & -2 & 0 \\ -2 & 2 & -1 \\ 0 & -1 & 2
\end{pmatrix}
    \quad\mbox{and}\quad
    \begin{pmatrix}
2 & -1 & 0 \\ -1 & 2 & -2 \\ 0 & -2 & 4
\end{pmatrix}
,
\]
which are also classical perfect matrices, so they are also arithmetically equivalent to
\(Q_{\mathsf{A}_3} \), but their neighborhood with respect to the copositive Ryshkov polyhedron
is different from that of \(Q_{\mathsf{A}_3} \).
In fact, their neighborhoods consist entirely of classical perfect matrices.
The remaining fifth contiguous copositive perfect matrix \(P_1 \) of \(Q_{\mathsf{A}_3} \)
reveals a~new phenomenon not observable for \(n=2 \): The matrix \(P_1 \) is positive semidefinite,
but not positive definite, as it has rank $2$.
It is part of a~series of such perfect copositive matrices which have a~growing number of
minimal vectors -- another phenomenon not observable for \(n=2 \):
For \(k \geq 1 \) let \(\alpha=k^2 +k+1 \), \(\beta=2k+1 \)
and
\[
	P_k =
\begin{pmatrix}
		2 & -\beta & 2 \\
		-\beta & 2\alpha & -\beta \\
		2 & -\beta & 2
\end{pmatrix}
.
\]
The matrices \(P_k \) are all positive semidefinite and have rank $2$.
We shall see that they can be constructed from the \(2 \times 2 \) perfect copositive matrices
\[
    \begin{pmatrix}
2\alpha & -\beta \\ -\beta & 2
\end{pmatrix}
\]
with minimal vectors
\(
\begin{pmatrix}
0 \\ 1
\end{pmatrix}
\),
\(
\begin{pmatrix}
1 \\ k
\end{pmatrix}
\) and
\(
\begin{pmatrix}
1 \\ k+1
\end{pmatrix}
\)
in the following way.
\begin{lemma}\label{lem:lifting}
	Let \(Q =
\begin{pmatrix}
M & m \\ m^\intercal & \mu
\end{pmatrix}
\in \copn \) be perfect
    copositive with \(M\in\mathcal{COP}^{n-1} \), \(m\in\IR^{n-1} \) and \(\mu\in\IRgeo \).
	If there is an \(x \in \MinC(Q) \) with \(x_n \geq 2 \), then the matrix
	
\[
		\widetilde{Q} =
\begin{pmatrix}
							M & m & m \\
							m^\intercal & \mu & \mu \\
							m^\intercal & \mu & \mu \\
						
\end{pmatrix}
\in\cop^{n+1}
\]
	is perfect copositive.
\end{lemma}
Note that the matrices \(P^\intercal\widetilde{Q}P \), for a~permutation matrix $P$,
are perfect copositive as well.
Thus applying Lemma~\ref{lem:lifting} to the above \(2\times2 \) matrices and multiplying with a~
suitable permutation matrix as above yields the matrices \(P_k \).

For the proof we need a~result from the classical theory, applied to strictly copositive matrices:
\begin{proposition}[{\cite[Proposition 3.5.3]
{martinet-2002}}]\label{prop:perfectinhyperplane}
    Let \(Q\in\interior\copn \).
    If there is a~hyperplane \(H\subset \IRn \), such that
	\(\operatorname{rank}\{ xx^\intercal :\ x\in\MinC Q \cap H \}=\binom{n}{2} \),
	then $Q$ is perfect copositive if and only if \(\MinC Q \setminus H \) spans \(\IRn \).
\end{proposition}
\begin{proof}[Proof of Lemma~\ref{lem:lifting}]
	To begin with, we represent the minimal vectors of \(\widetilde{Q} \) in terms of the minimal
	vectors of $Q$.
	Let \(x\in\IR^{n+1} \) and let \(\hat{x}\in\IR^{n-1} \) be the vector made up of the first
	\(n-1 \) entries of $x$.
	Then we have
	
\[
	 \widetilde{Q}[x] = M[\hat{x}] + 2(x_n+x_{n+1})m^\intercal\hat{x} + {(x_n+x_{n+1})}^2 \mu
		= Q\left[
\begin{pmatrix}
\hat{x} \\ x_n + x_{n+1}
\end{pmatrix}
\right].
\]
	Therefore
	
\begin{align*}
		& \MinC(\widetilde{Q}) = \\
			& \big\{
			 {(x_1,\dots, x_{n-1}, a, b)}^\intercal :\
			 x\in\MinC(Q),\ a,b\in\IZ_{\geq0} \text{ and } a+b=x_n
			\big\}.
\end{align*}
    Note that the set of rank-$1$-matrices from the subset
    \[
        \big\{ {(x_1, \dots, x_n, 0)}^\intercal :\ x\in\MinC Q \big\}
    \]
    has dimension \(\binom{n+1}{2} \), since $Q$ is perfect copostive.
    Hence, with Proposition~\ref{prop:perfectinhyperplane}, \(\widetilde{Q} \) is perfect
    copositive if the remaining minimal vectors span \(\IR^{n+1} \).
	But this is exactly the case when $Q$ has a~minimal vector $x$ with \(x_n \geq 2 \), since
	then \(\widetilde{Q} \) has at least one minimal vector of the form
	\({(x_1, \dots, x_{n-1}, a, b)}^\intercal \) with \(a, b \neq 0 \).
	Furthermore, \(\widetilde{Q} \) has at least \(\binom{n+1}{2} \) minimal vectors of the form
	\({(x_1, \dots, x_{n-1}, 0, x_n)}^\intercal \).
	Both types are not contained in the hyperplane \(\{ x_{n+1} = 0 \} \) and together span
	\(\IR^{n+1} \).
\end{proof}

Besides the perfect copositiveness of the matrices of the form \(P_k \), we also obtain the following results.
\begin{corollary}\label{cor:pkresults}
	The following conditions hold:
\begin{enumerate}
		\item\label{pkresultsone} \(\minC P_k=2 \).
		\item\label{pkresultstwo} \(\MinC P_k=
                \left(\left\{ e_1,
\begin{pmatrix}
k \\ 1 \\ 0
\end{pmatrix}
,
\begin{pmatrix}
k + 1 \\ 1 \\ 0
\end{pmatrix}
\right\}+
                 \left\langle
\begin{pmatrix}
-1 \\ 0 \\ 1
\end{pmatrix}
\right\rangle\right)\cap\IZngeo
               \)
               and in addition, \\ \(\left\vert \MinC P_k \right\vert=2k+5 \).
		\item\label{pkresultsthree} \(P_k \) and \(P_{k+1} \) are contiguous copositive perfect matrices.
	
\end{enumerate}
\end{corollary}
\begin{proof}
	Looking at the proof of Lemma~\ref{lem:lifting}, we see that the lifting operation preserves
	the copositive minimum, i.e.\ \(\minC Q = \minC\widetilde{Q} \), hence \eqref{pkresultsone} follows.
	For \eqref{pkresultstwo}, we see that the minimal vectors of \(\widetilde{Q} \) are of the form
	
\[
		{(x_1, \dots, x_{n-1}, x_n-l, l)}^\intercal \text{ for } l=0,\dots,x_n
\]
	with \(x\in\MinC Q \).
	But this results exactly in the above representation, when plugging in the minimal vectors
	of the \(2\times2 \) matrices.
	A straightforward calculation shows that
	
\[
		\widetilde{R}=
\begin{pmatrix}
0 & -2 & 0 \\ -2 & 4k+4 & -2 \\ 0 & -2 & 0
\end{pmatrix}
\]
	is an extreme ray of \({\mathcal{V}(P_k)}^* \).
	Further \(P_{k+1} = P_k + \widetilde{R} \), and therefore \eqref{pkresultsthree} holds.
\end{proof}
Looking at the minimal vectors of \(P_k \), we see that the conditions of Lemma~\ref{lem:lifting}
are again satisfied.
Hence, we can apply Lemma~\ref{lem:lifting} to the \(P_k \) to obtain a~series of \(4\times 4 \)
positive semidefinite perfect copositive matrices, with similar properties as the \(P_k \).
Iterating this procedure, we confirm the existence of \(n\times n \) positive semidefinite
perfect copositive matrices.
Note, however, that the rank does not increase with this lifting method, hence all matrices obtained
from \(2\times 2 \) matrices will have rank $2$.
But we can easily apply Lemma~\ref{lem:lifting} to suitable \(k\times k \) positive definite
perfect copositive matrices, to obtain matrices of rank $k$.
We obtain
\begin{corollary}
    For any \(n\geq 3 \) there exist positive semidefinite matrices of rank \(1< k <n \),
    which are perfect copositive.
\end{corollary}
Note that, while Lemma~\ref{lem:lifting} cannot be used to construct all of them,
it has the advantage that it can be applied to more than just positive definite matrices.

For example, looking at the neighborhood of \(P_1 \) we find the indefinite perfect copositive matrix
\[
    I=
\begin{pmatrix}
2 & -5 & 4 \\ -5 & 14 & -9 \\ 4 & -9 & 6
\end{pmatrix}
\]
with \(\minC I=2 \) and
\[
    \MinC I=
    \left\{
    \begin{pmatrix}
0 \\ 1 \\ 1
\end{pmatrix}
,
    \begin{pmatrix}
0 \\ 1 \\ 2
\end{pmatrix}
,
    \begin{pmatrix}
0 \\ 2 \\ 3
\end{pmatrix}
,
    \begin{pmatrix}
1 \\ 0 \\ 0
\end{pmatrix}
,
    \begin{pmatrix}
2 \\ 1 \\ 0
\end{pmatrix}
,
    \begin{pmatrix}
3 \\ 1 \\ 0
\end{pmatrix}
,
    \begin{pmatrix}
1 \\ 1 \\ 1
\end{pmatrix}
    \right\}.
\]
Since \(I \) is neither positive semidefinite nor nonnegative, and \(n=3 \), \(I \) lies in the set
\((\mathcal{S}^{3}_{\geq0}+\mathcal{N}^{3})\setminus(\mathcal{S}^{3}_{\geq 0}\cup\mathcal{N}^3) \).
Obviously, the conditions of Lemma~\ref{lem:lifting} are satisfied for \(I \).
We further have the following result concerning Lemma~\ref{lem:lifting}.
\begin{lemma}\label{lem:closedlifting}
	Under the lifting operation described in Lemma~\ref{lem:lifting} the components
	
\[
		(\sngeo+\nonn)\setminus(\sngeo\cup\nonn) \text{ and } \copn\setminus(\sngeo+\nonn)
\]
	are closed, as $n$ goes to \(n + 1 \).
	In other words, lifting matrices in \((\sngeo+\nonn)\setminus(\sngeo\cup\nonn) \) gives matrices in
    \((\mathcal{S}^{n+1}_{\geq0}+\mathcal{N}^{n+1})\setminus
	 (\mathcal{S}^{n+1}_{\geq0}\cup\mathcal{N}^{n+1}) \).
\end{lemma}
\begin{proof}
    Follows immediately from the fact that \(\sngeo \) and \(\nonn \) are closed under
	the lifting operation, as $n$ goes to \(n+1 \).
\end{proof}
Thus applying Lemma~\ref{lem:lifting} to $I$ and the resulting matrices, we obtain
\begin{corollary}
    For any \(n \geq 3 \) there are perfect copositive matrices in
    \((\sngeo+\nonn)\setminus(\sngeo\cup\nonn) \).
\end{corollary}

The last component we look at is \(\copn\setminus(\sngeo+\nonn) \).
Matrices lying in this component are called \textit{exceptional} copositive matrices
(a term coined in~\cite{jr-2008}).
They only exist for \(n \geq 5 \).
As explained in Section~\ref{sec:copositive}, for a~given matrix $Q$, copositive matrices can be
used as a~certificate for \(Q\notin\cpn \).
If $Q$ is doubly nonnegative (i.e.\ positive semidefinite and nonnegative), a~certificate matrix
$P$ has to be exceptional since (see~\cite[Theorem 1.35 (e)]{bs-2003})
\[
	{(\sngeo+\nonn)}^* = {(\sngeo)}^* \cap{(\nonn)}^* = \sngeo\cap\nonn
	\text{ and }
	{(\sngeo\cap\nonn)}^* = \sngeo+\nonn.
\]
Hence any non-exceptional copositive matrix has a~nonnegative inner product with the doubly
nonnegative $Q$ and therefore cannot be a~certificate matrix.
By traversing the neighborhood graph for \(n = 5 \) we find for the doubly nonnegative but not
completely positive matrix
\[
    \begin{pmatrix}
        1 & 1 & 0 & 0 & 1 \\
        1 & 2 & 1 & 0 & 0 \\
        0 & 1 & 2 & 1 & 0 \\
        0 & 0 & 1 & 2 & 1 \\
        1 & 0 & 0 & 1 & 6
    \end{pmatrix}
\]
the exceptional perfect copositive certificate matrix
\[
    E = \frac13
\begin{pmatrix}
       366 & -300 & 197 & 147 & -81 \\
       -300 & 246 & -161 & 123 & 69 \\
       197 & -161 & 106 & -82 & 39 \\
       147 & 123 & -82 & 66 & -33 \\
       -81 & 69 & 39 & -33 & 18
    \end{pmatrix}
\]
with 18 minimal vectors
\begin{align*}
    \MinC E = & \left\{
        \begin{pmatrix}
1 \\ 0 \\ 0 \\ 0 \\ 4
\end{pmatrix}
,
        \begin{pmatrix}
2 \\ 0 \\ 0 \\ 0 \\ 9
\end{pmatrix}
,
        \begin{pmatrix}
1 \\ 0 \\ 0 \\ 0 \\ 5
\end{pmatrix}
,
        \begin{pmatrix}
1 \\ 0 \\ 0 \\ 1 \\ 6
\end{pmatrix}
,
        \begin{pmatrix}
1 \\ 2 \\ 1 \\ 0 \\ 0
\end{pmatrix}
,
        \begin{pmatrix}
0 \\ 0 \\ 1 \\ 2 \\ 2
\end{pmatrix}
,
        \begin{pmatrix}
0 \\ 0 \\ 2 \\ 4 \\ 3
\end{pmatrix}
,
        \begin{pmatrix}
0 \\ 0 \\ 1 \\ 2 \\ 1
\end{pmatrix}
,
        \begin{pmatrix}
0 \\ 2 \\ 4 \\ 1 \\ 0
\end{pmatrix}
,\right.\\
        &\left.
\begin{pmatrix}
0 \\ 0 \\ 0 \\ 1 \\ 2
\end{pmatrix}
,
        \begin{pmatrix}
5 \\ 6 \\ 0 \\ 0 \\ 0
\end{pmatrix}
,
        \begin{pmatrix}
0 \\ 1 \\ 3 \\ 2 \\ 0
\end{pmatrix}
,
        \begin{pmatrix}
2 \\ 0 \\ 0 \\ 1 \\ 11
\end{pmatrix}
,
        \begin{pmatrix}
2 \\ 3 \\ 1 \\ 0 \\ 0
\end{pmatrix}
,
        \begin{pmatrix}
0 \\ 3 \\ 6 \\ 2 \\ 0
\end{pmatrix}
,
        \begin{pmatrix}
0 \\ 2 \\ 3 \\ 0 \\ 0
\end{pmatrix}
,
        \begin{pmatrix}
1 \\ 1 \\ 0 \\ 0 \\ 1
\end{pmatrix}
,
        \begin{pmatrix}
4 \\ 5 \\ 0 \\ 0 \\ 0
\end{pmatrix}
    \right\}.
\end{align*}
Applying Lemma~\ref{lem:lifting} and~\ref{lem:closedlifting} to $E$ and the resulting
matrices gives us
\begin{corollary}
    For any \(n \geq 5 \) there are perfect copositive matrices in
    \(\copn\setminus(\sngeo+\nonn) \).
\end{corollary}

\section{Embedding of the classical perfect matrices}\label{sec:embedding}

As we have seen in the previous sections, there are classical perfect matrices which are also
perfect copositive but there are also those which aren't perfect copositive.
However, we also know that the classic Ryshkov polyhedron \(\sR \) is a~subset of the copositive
Ryshkov polyhedron \(\sR_\cop \).
Furthermore, the vertices of \(\sR \) are exactly the classically perfect matrices
(the vertices of \(\sR_\cop \) being the perfect copositive ones).
Hence \(\sR \) and \(\sR_\cop \) share some vertices.
Since \(\sR \) only has finitely many, up to arithmetical equivalence,
this suggests the following question:
Does every orbit of classically perfect matrices contain at least one,
which is also perfect copositive?
\begin{theorem}\label{thm:embedding}
	Every classically perfect matrix \(Q\in\sngo \) is arithmetically equivalent
    to a~perfect copositive matrix, i.e.\ there exists a~\(U\in\glnz \),
	such that \(U^\intercal QU \) is perfect copositive.
\end{theorem}
For the proof we need a~few tools from the classical theory,
namely the concept of a~\emph{Minkowski reduced} matrix.
We say that a~matrix \(Q=(q_{ij})\in\sngo \) is Minkowski reduced, if
\[
	Q[v]\geq Q[e_i]=q_{ii} \text{ for all } v\in\IZnmo \text{ such that } \gcd(v_i,\dots,v_n)=1.
\]
The two properties of Minkowski reduced matrices we need are (1) that every classically perfect
matrix is arithmetically equivalent to a~Minkowski reduced perfect matrix (which can be found in~\cite[Theorem 2.2]{schuermann-2009}),
and (2) that \(\Min Q \) lies entirely in some cube, i.e.\ there is an integer $q$,
depending only on the dimension, such that for all \(x\in\Min Q \) the bound \(|x_j| < q \) holds
for \(j=1, \dots, n \) (see~\cite[Section 2.2]{schuermann-2009}).

We continue with the proof.

\begin{proof}[Proof of Theorem~\ref{thm:embedding}]
We may assume without loss of generality that $Q$ is Minkowski reduced.
As for any \(U\in\glnz \) we have the relation
\(\Min(U^\intercal QU)=U^{-1}\Min(Q) \),
it suffices to show that we can choose $U$ such that
\(U^{-1}\Min(Q)\in\pm\IZngeo \).
For this, we choose the lower triangular matrix \(U^{-1}=(u_{ij}) \)
with \(u_{ij}=q^{i-j} \) if \(i \geq j \) and \(u_{ij}=0 \) otherwise.

Because \(|x_j|<q \) for any \(x\in\Min Q \), if \(x_1>0 \) we have that \(U^{-1}x\in\IZngeo \),
since
\[
    {(U^{-1}x)}_i=\sum_{j=1}^{i}q^{i-j}x_j \geq q^{i-1}x_1 - \sum_{j=2}^{i}q^{i-j}|x_j|
    \geq q^{i-1} - \sum_{j=2}^{i}q^{i-j}(q-1) = 1
\]
for \(i=1,\ldots, n \).
Similarly, we find \(U^{-1}x\in-\IZngeo \), if \(x_1<0 \).
In case \(x_1=0 \), the assertion follows from a~corresponding argument for \(n-1 \),
respectively for \(n-k \) with \(x_{k}\neq0 \) and \(x_j=0 \) for \(j<k \).
\end{proof}

\section{Polytopal Representations}\label{sec:polyrepr}
In this section we want to look at the copositive Ryshkov polyhedron \(\sR_\cop \) more closely,
describe parts of its structure and in doing so, prove some general results on locally finite polyhedra,
which will turn out to be generalizations of results from the theory of polyhedra.
Since both \(\sR \) and \(\sR_\cop \) are of similar structure, the results in this section
will also be used to compare the two.
Recall that the classical Ryshkov polyhedron \(\sR \) is equal to the convex hull of the classical
perfect matrices~\cite[Chapter 3.1]{schuermann-2009}:
\[
	\sR = \conv\big\{ P\in\sn :\ P \text{ perfect, } \min(P)=1 \big\}.
\]
But as seen in Section~\ref{sec:nequal2}, there is no such representation in the copositive case.

The main result of this section generalizes the polytopal representation~\eqref{eq:rcop-2-poly}
to arbitrary \(n \geq 2 \):
\begin{theorem}\label{thm:rcop-poly}
	Let \(E_{ij}=\big(e_ie_j^\intercal + e_je_i^\intercal\big)\in\sn \) for
	\(1 \leq i,j \leq n \).
	Then \(\sR_\cop \) can be represented as
	
\begin{align*}
		\sR_\cop =& \conv\big\{ P \text{ perfect copositive}: \ \minC(P)=1 \big\} \\
			 	 +& \cone\big\{ E_{ij} : \ i \neq j \big\} \nonumber
\end{align*}
\end{theorem}
For the proof we use results about so-called locally finite polyhedra.
While Theorem~\ref{thm:rcop-poly} takes place in \(\sn \), these results hold in any
Euclidean space $E$ with inner product \(\langle \cdot,\cdot \rangle \).
Hence we will move over to the general setting.
Because we are not aware of a~reference for these basic results, we present them here.

First we recall some basic notions about polyhedra.
For further reading we refer to~\cite{barvinok-2002} and~\cite{ziegler-1995}.
For a~hyperplane \(H(a, \beta) = \{x\in E:\ \langle a, x \rangle = \beta \} \), with normal vector
\(a\in E \), we call the set
\(H^{\leq}(a, \beta)= \{ x\in E:\ \langle a, x \rangle \leq \beta \} \) a~halfspace (induced by $H$).
Then a~set \(\polyh \subseteq E \) is a~polyhedron if it is an intersection of finitely many halfspaces.
It is a~polytope if it is bounded, or equivalently if it is the convex hull of finitely many points.
A face of a~polyhedron \(\polyh \) (resp.\ closed convex set) is a~subset \(\pfacef \) of
\(\polyh \) for which there exists a~hyperplane $H$ such that \(\polyh \) is contained entirely
in one of the halfspaces generated by $H$ and \(\pfacef=\polyh\cap H \).
Such a~hyperplane is also called a~\emph{supporting hyperplane} of \(\polyh \).
The $0$-dimensional faces, or $0$-faces, of \(\polyh \) are called vertices, the $1$-faces
edges, and the \(\dim(\polyh)-1 \)-faces facets.
An extreme ray is an unbounded edge with a~vertex.
Two nice properties of faces of polyhedra are that a~face \(\pfacef \) of a~polyhedron \(\polyh \)
is a~polyhedron as well, and that if \(\pfaceg \) is a~face of \(\pfacef \),
then \(\pfaceg \) is also a~face of \(\polyh \).

An intersection of countably many halfspaces \(\lfp = \bigcap_{i=1}^{\infty}H_i^{\leq}\subseteq E \)
is a~\emph{locally finite polyhedron}, if the intersection of \(\lfp \) with an arbitrary polytope
is itself a~polytope.
In~\cite[Theorem 3.1]{schuermann-2009} and~\cite[Lemma 2.4]{dsv-2017} it was shown, that both
\(\sR \) and \(\sR_\cop \) are locally finite polyhedra.
Note that polyhedra themselves are also locally finite polyhedra.

To prove Theorem~\ref{thm:rcop-poly} we first need some properties of faces of locally finite
polyhedra, which generalize the properties of faces of polyhedra mentioned above.
\begin{lemma}\label{lem:faceanalog}
	Let \(\lfp=\bigcap_{i=1}^{\infty}H^{\leq}(a_i,\beta_i) \) be a~locally finite polyhedron.
	Then \(\lfpfacef \) is a~face of \(\lfp \) if and only if there is a~finite index set $I$,
	such that \(\lfpfacef=\lfp\cap\bigcap_{i\in I}H(a_i,\beta_i) \).
\end{lemma}
\begin{proof}
	Note that \(H=\left\{ x:\ \langle\sum_{i\in I}a_i,x\rangle = \sum_{i\in I}\beta_i \right\} \)
	is a~supporting hyperplane of \(\lfp \).
    As \(\lfpfacef=\lfp\cap H \), we see that \(\lfpfacef \) is a~face of \(\lfp \).

	On the other hand, let \(\lfpfacef \) be a~face of \(\lfp \).
	Then there is some hyperplane $H$ with \(\lfpfacef = \lfp\cap H \) and
    \(\lfp\subseteq H^{\leq} \).
	By choosing $x$ in the relative interior of \(\lfpfacef \), we find an affine subspace $A$
	with \(A \subseteq H \) and \(\lfpfacef = \lfp\cap A \).
	Now take some arbitrary polytope $C$ containing $x$ in its interior.
	Then, the intersection \(C\cap\lfp \) is a~polytope by assumption, and further
	
\[
		C\cap\lfpfacef=C\cap\lfp\cap A \text{ is a~face of } C\cap\lfp.
\]
	As it is a~polytope it can be described as an intersection of finitely many halfspaces, some of
	them among the \(H^{\leq}(a_i, \beta_i) \).
	However the halfspaces with bounding hyperplanes containing $A$ can be chosen solely from the
	\(H^{\leq}(a_i, \beta_i) \).
	In other words
	
\[
		A = \bigcap_{i\in I}H(a_i, \beta_i) \text{ with a~suitable finite index set } I.
\]
	Therefore
	\(\lfpfacef = \lfp\cap\bigcap_{i\in I}H(a_i, \beta_i). \)
\end{proof}
Note that from the proof of Lemma~\ref{lem:faceanalog} the index set $I$ can always be chosen such
that \(|I| = \operatorname{co-dim}\lfpfacef \).
The representation of faces in Lemma~\ref{lem:faceanalog} is essentially the same as in the
polyhedral case, cf.~\cite[Section 8.3]{schrijver-1986}.
As a~consequence, we obtain
\begin{corollary}\label{cor:faceoflfp}
	Let \(\lfp \) be a~locally finite polyhedron, \(\lfpfacef \) be a~face of \(\lfp \).
	Then the following holds:
	
\begin{enumerate}
		\item \(\lfpfacef \) is a~locally finite polyhedron
		\item If \(\lfpfaceg \) is a~face of \(\lfpfacef \),
			then \(\lfpfaceg \) is also a~face of \(\lfp \)
	
\end{enumerate}
\end{corollary}
\begin{proof}
    It suffices to show that the intersection \(\lfpfacef\cap\polyh \) for some arbitrary polytope
	\(\polyh \) is again a~polytope.
    Applying Lemma~\ref{lem:faceanalog} to \(\lfpfacef \) we see that
	\(\lfpfacef=\lfp\cap\bigcap_{i\in I}H(a_i,\beta_i) \) for some finite index set $I$.
    Hence
    \[
        \lfpfacef\cap\polyh=\lfp\cap\bigcap_{i\in I}H(a_i,\beta_i)\cap\polyh=
		(\lfp\cap\polyh)\cap\left(\bigcap_{i\in I}H(a_i,\beta_i)\cap\polyh\right).
    \]
    \(\lfp \) is a~locally finite polyhedron, hence \(\lfp\cap\polyh \) is a~polytope.
    The intersection of a~polytope with finitely many affine hyperplanes is a~polytope,
	hence \(\bigcap_{i}H(a_i,\beta_i)\cap\polyh \) is also a~polytope.
    Finally, intersections of polytopes are polytopes, and so \(\lfpfacef\cap\polyh \) is a~polytope.

    Now applying Lemma~\ref{lem:faceanalog} to \(\lfpfaceg \), we get
	
\[
		\lfpfaceg=\lfpfacef\cap\bigcap_{j\in J}H(a_j,\beta_j)=
		\lfp\cap\bigcap_{i\in I}H(a_i,\beta_i)\cap\bigcap_{j\in J}H(a_j,\beta_j)=
		\lfp\cap\bigcap_{k\in K}H(a_k,\beta_k)
\]
	for suitable finite index sets $J$ and $K$.
    Therefore, \(\lfpfaceg \) is a~face of \(\lfp \).
\end{proof}

It is known that a~polyhedron \(\polyh \) can be decomposed into a~polytope \(\mathsf{Q} \) and a~
polyhedral cone \(\mathsf{C} \), such that \(\polyh=\mathsf{Q}+\mathsf{C} \).
This representation follows from the Farkas-Minkowski-Weyl theorem, see~\cite[Section 7.2]{schrijver-1986}.
Because of Corollary~\ref{cor:faceoflfp} a~similar decomposition also holds for locally finite polyhedra
(in fact such a~decomposition holds for all convex sets with these facial properties,
cf.~\cite[Theorem 18.5, 18.7]{rockafellar-1970}).
\begin{proposition}\label{prop:lfpmw}
	Let \(\lfp \) be a~locally finite polyhedron with a~vertex.
	Then
	
\begin{align}\label{eq:lfprepresentation}
			\lfp =& \conv\big\{ V:\ V \text{ is a~vertex of } \lfp \big\} \\
			 +& \cone\big\{ R :\ R \text{ is an extreme ray of } \lfp \big\} \nonumber
\end{align}
\end{proposition}
To prove Theorem~\ref{thm:rcop-poly} via Proposition~\ref{prop:lfpmw} it now suffices to investigate
the vertices and extreme rays of \(\sR_\cop \).
From our previous discussion, we already know that the vertices of \(\sR_\cop \) are the perfect
copositive matrices $P$ with \(\minC P=1 \).
Therefore, we only need to construct the extreme rays of locally finite polyhedra,
and especially those of \(\sR_\cop \).

For this, let the recession cone of a~closed convex set $S$ be the set
\[
	\mathsf{C}_S = \big\{ R :\ Q+\lambda R \in S \text{ for all } Q \in S
        \text{ and all } \lambda\geq0
                   \big\}.
\]
\begin{lemma}\label{lem:lfp-rec}
	Let \(\lfp \) be a~locally finite polyhedron with a~vertex.
	The extreme rays of \(\lfp \) are in $1$-to-$1$ correspondence with the extreme rays of the
	recession cone \(\mathsf{C}_\lfp \) of \(\lfp \), which correspond to the facets of the dual
	of the recession cone.
\end{lemma}
\begin{proof}
    We may assume that \(\dim\lfp=n \).
	The extreme rays of \(\lfp \) are of the form
	
\[
		\lfpfacef=\big\{ V+\lambda R :\ \lambda\geq0 \big\}
\]
	for some vertices $V$ and directions $R$ of \(\lfp \).
	It follows that $R$ is an extreme ray of the recession cone \(\mathsf{C}_\lfp \) (see~\cite[p. 163]{rockafellar-1970}).
	By Lemma~\ref{lem:faceanalog} there exist \(i_1,\dots,i_{n-1} \), such that
	\(\lfpfacef=\lfp\cap\bigcap_{j=1}^{n-1}H(a_{i_j},\beta_{i_j}) \) with
	\(\text{dim}\{ a_{i_1},\dots,a_{i_{n-1}} \}=n-1 \).
	By relabeling if necessary, we can assume that \((i_1, \dots, i_{n-1})=(1,\dots,n-1) \).
	For those indices $i$ we have
	
\[
		\langle V+\lambda R, a_i \rangle = \beta_i \text{ for all } \lambda\geq0.
\]
	Hence \(\langle R, a_i \rangle = 0 \) for \(i=1, \dots, n-1 \).
	Note further that \(\langle a_i,S \rangle \leq 0 \) holds for all $i$ and \(S\in C_\lfp \),
	and so \(-a_i\in{\mathsf{C}_\lfp}^* \).
    Thus \(-a_1, \dots, -a_{n-1} \) are contained in the face
	
\[
		\lfpfacef^* =\{ a\in {\mathsf{C}_\lfp}^* :\ \langle R, a~\rangle=0 \}
\]
	of \({\mathsf{C}_\lfp}^* \).
	Therefore \(\lfpfacef^* \) is a~facet of \({\mathsf{C}_\lfp}^* \) with normal vector $R$.
\end{proof}
Recall that \({(\copn)}^* = \cpn \).
A straightforward calculation shows that \(\mathsf{C}_{\sR_\cop} = \copn \),
hence by Lemma~\ref{lem:lfp-rec} we have to look at the facets of \(\cpn \).
They were already investigated in~\cite{dickinson-2013a}.
\begin{lemma}[{\cite[Theorem 8.33]
{dickinson-2013a}}]\label{lem:dickinson}
    The normal vectors of the facets of \(\cpn \) are the matrices \(E_{ij} \) with \(i\neq j \).
\end{lemma}
Altogether we obtain the proof of the main result of this section:
\begin{proof}[Proof of Theorem~\ref{thm:rcop-poly}]
	Applying Proposition~\ref{prop:lfpmw} we get
	
\begin{align*}
		\sR_\cop =& \conv\big\{ P:\ P \text{ is a~vertex of } \sR_\cop \big\} \\
				 +& \cone\big\{ R :\ R \text{ is a~extreme ray of } \sR_\cop \big\}.
\end{align*}
	The rest follows from Lemma~\ref{lem:lfp-rec} and~\ref{lem:dickinson}.
\end{proof}

To end this section, we shortly discuss an application:
In~\cite[Section 2.2]{dsv-2021} it was shown that the interior of \(\copn \) satisfies
\begin{equation}\label{eq:conesrequalintcop}
	\interior\copn = \cone\sR_\cop\setminus\{0\}
\end{equation}
(which generalizes the classical case, where we have \(\interior\sngeo=\cone\sR\setminus\{0\} \)).
Hence Theorem~\ref{thm:rcop-poly} now gives us a~kind of discretization of \(\copn \)
in terms of the vertices and rays of \(\sR_\cop \).
Using duality, we find a~new and interesting representation of \(\cpn \) in terms
of perfect copositive matrices.
\begin{corollary}
\[
	\cpn=\big\{
	 Q\in\nonn: \ \left<P,Q\right>\geq0 \text{ for all perfect copositive matrices $P$}
	\big\}
\]
\end{corollary}

\section*{Acknowledgement}

We like to thank Mathieu Dutour Sikiri\'{c}, Frieder Ladisch, Robert Sch\"{u}ler
and the anonymous referee for helpful suggestions.
Both authors gratefully acknowledge support by the German Research Foundation (DFG)
under grant SCHU 1503/8-1.

%%% REFERENCES %%%

\EditInfo{March 31, 2023}{June 29, 2023}{Camilla Hollanti and Lenny Fukshansky}

\end{document}